\documentclass[12pt]{amsart}
\usepackage{amsfonts, amsmath}

\usepackage{ amsmath, amsthm, amsfonts, amssymb, color}
 \usepackage{mathrsfs}
 \usepackage{amsmath,amstext,amsthm,amssymb,amsxtra}

\theoremstyle{plain}
\newtheorem{theorem}{Theorem}[section]

\newtheorem{proposition}[theorem]{Proposition}
\newtheorem{lemma}[theorem]{Lemma}
\newtheorem{corollary}[theorem]{Corollary}

\theoremstyle{definition}
\newtheorem{definition}[theorem]{Definition}
\newtheorem{remark}[theorem]{Remark}

\allowdisplaybreaks

\def\RR{\mathbb R}
\def\R{\mathbb R}

\def\11{1\!\!1}
\def\supp{\text{supp}}
\def\SL{\sqrt{L}}

\newcommand\D{\mathcal{D}}
\def\DDelta{\overrightarrow{\Delta}}
\def\f12{\frac{1}{2}}
\def\forms{\Lambda^1T^*M}

\DeclareMathOperator{\support}{supp}

\title[Riesz transforms on non-compact manifolds ]
{Riesz transforms on non-compact manifolds}
\author{Peng Chen, Jocelyn Magniez and El Maati Ouhabaz}

\address{Peng Chen,  Institut de Math\'ematiques de Bordeaux,  Univ.  Bordeaux, UMR 5251,
351, Cours de la Lib\'eration 33405 Talence, France}
\email{peng.chen@math.u-bordeaux1.fr}

\address{Jocelyn Magniez,  Institut de Math\'ematiques de Bordeaux,  Univ.  Bordeaux, UMR 5251,
351, Cours de la Lib\'eration 33405 Talence, France}
\email{Jocelyn.Magniez@math.u-bordeaux1.fr}

\address{El Maati Ouhabaz,  Institut de Math\'ematiques de Bordeaux,  Univ.  Bordeaux, UMR CNRS 5251,
351, Cours de la Lib\'eration 33405 Talence, France}
\email{Elmaati.Ouhabaz@math.u-bordeaux1.fr}
\thanks{The research of the  authors was partially supported by the ANR project  HAB, ANR-12-BS01-0013-02.}

\medskip


  \date{\today}

\begin{document}

\begin{abstract}Let $M$ be a complete non-compact Riemannian manifold satisfying the doubling volume property
 as well as a Gaussian upper bound for the  corresponding heat kernel. We study the boundedness of the Riesz transform $d\Delta ^{-\frac{1}{2}}$ on both Hardy spaces $H^p$ and Lebesgue spaces $L^p$  under two different conditions on the negative part of the Ricci curvature $R^-$.
 First we prove that if $R^-$ is $\alpha$-subcritical for some $\alpha \in [0,1)$, then the Riesz transform $d^*\DDelta^{-\frac{1}{2}}$ on differential $1$-forms is bounded from the associated Hardy space  $H^p_{\overrightarrow{\Delta}}(\Lambda^1T^*M)$ to $L^p(M)$ for all $p\in [1,2]$.
 As a consequence,   $d\Delta^{-\frac{1}{2}}$ is bounded on $ L^p$ for all $p\in (1,p_0)$ where $p_0>2$ depends on $\alpha$ and the  constant appearing in the doubling  property. Second, we prove that if
$$\int_0^1 \left\|\frac{|R^-|^{\frac{1}{2}}}{v(\cdot,\ \sqrt{t})^{\frac{1}{p_1}}}\right\|_{p_1}\frac{dt}{\sqrt{t}}+\int_1^\infty \left\|\frac{|R^-|^{\frac{1}{2}}}{v(\cdot,\ \sqrt{t})^{\frac{1}{p_2}}}\right\|_{p_2}\frac{dt}{\sqrt{t}}<\infty,$$
for some $p_1>2$ and $p_2>3$,  then the  Riesz transform $d\Delta^{-\frac{1}{2}}$
is bounded on $L^p$ for all $1<p<p_2$. In the particular case where
$v(x, r) \ge C r^D$  for all $r \ge 1$ and  $|R^-| \in L^{D/2 -\eta} \cap L^{D/2 + \eta}$ for some $\eta > 0$, then $d\Delta^{-\frac{1}{2}}$ is bounded on $L^p$ for all $1<p< D.$

Furthermore, we study the boundedness of the Riesz transform of Schr\"odinger operators $A=\Delta+V$ on $L^p$ for $p>2$ under conditions on $R^-$ and the potential $V$.  We prove both positive and negative results on the boundedness of $dA^{-\frac{1}{2}}$ on $L^p$.  \end{abstract}

\maketitle


\bigskip

   \section{Introduction  }
\setcounter{equation}{0}

Let $(M,g)$ be a complete non-compact Riemannian manifold and let
$\rho$ be the geodesic distance and $\mu$ be the Riemannian measure
associated with the metric $g$. Assume that $M$ satisfies the
doubling volume property, that is, there exists a constant $C>0$
such that
$$v(x,2 r)\leq C v(x,r) {\mbox{ for all }} x\in M,\,
r\geq 0,$$
where $v(x,r)$ denotes the volume of the ball $B(x,r)$ of center $x$
and radius $r$. This property is equivalent to the following one. There exist  constants $C>0$ and $ D>0$ such that
\begin{equation}\label{def:doubling2}
v(x,\lambda r)\le C\lambda^{D}v(x,r) \;{\mbox{ for all }} x\in M,\, r\ge 0, \,\lambda\ge 1.
\end{equation}
In the sequel we will be interested in the  smallest possible $D$ for which  \eqref{def:doubling2} holds. For convenience we call $D$  ``the"  doubling dimension or the homogeneous dimension. 

Let $\Delta$ be the non-negative Laplace-Beltrami operator on $M$
and $p_t(x,y)$  the corresponding heat kernel, i.e.,  the integral
kernel
 of the semigroup $e^{-t\Delta}$. We
assume that $p_t(x,y)$ satisfies a Gaussian upper
bound
\begin{eqnarray}\label{def:Gassianbound}
p_t(x,y)\leq C v(x,\sqrt{t})^{-1}e^{-c\frac{\rho^2(x,y)}{t}} {\mbox{
for all }} t>0, \, x,y\in M,
\end{eqnarray}
where $c, C>0$ are constants.  The validity of \eqref{def:Gassianbound} has been intensively studied in the literature.

We consider the  Riesz transform $d\Delta^{-\f12}$. Integration by
parts shows that  $d\Delta^{-\f12}$ is bounded  from $L^2(M)$ to
$L^2(\forms)$, where $\forms$ denotes the space of differential
$1$-forms. We address the problem whether the Riesz transform
$d\Delta^{-\f12}$ could be extended to a  bounded operator from
$L^p(M)$ to $L^p(\forms)$ for $p\neq 2$.

Under the assumptions \eqref{def:doubling2} and
\eqref{def:Gassianbound}, it was proved by Coulhon and Duong
\cite{CD} that $d\Delta^{-\f12}$ is bounded from $L^p(M)$ to
$L^p(\forms)$ for all $p\in (1,2]$.  They also gave a counter-example
which shows that   \eqref{def:doubling2} and
\eqref{def:Gassianbound} are not sufficient  in the case $p > 2$. So additional assumptions are needed.
 Many works have been devoted to this problem.

Under Li-Yau estimates, or equivalently under the doubling condition and a $L^2$ Poincar\'e inequality, Auscher and Coulhon \cite{AC} proved that there exists $\epsilon>0$ such that $d\Delta^{-\f12}$ is bounded on $L^p$ for all $2\le p<2+\epsilon$. In the same setting, Auscher, Coulhon, Duong and Hofmann \cite{acdh} found an equivalence between the boundedness of the Riesz transform on $L^p$ for $p>2$ and the gradient  estimate $\|de^{-t\Delta}\|_{p-p}\le C/\sqrt{t}$ for the corresponding semigroup on $L^p$. 

Bakry \cite{B} proved that if the manifold has a non-negative Ricci curvature, then $d\Delta^{-\f12}$ is bounded on $L^p$ for all $p\in(1,\infty)$.

Another idea to treat the case $p>2$ is by duality.  By  the
commutation formula $\DDelta d=d\Delta$, we are  then interested in
the Riesz transform $d^*\DDelta^{-\f12}$  where $\DDelta = d d^* +
d^* d$ is the Hodge-de Rham Laplacian  on differential $1$-forms.
The boundedness of $d^*\DDelta^{-\f12}$ on $L^q$ for some $q \in
(1,2)$ implies the boundedness of  $d\Delta^{-\f12}$ on $L^p$ where
$\frac{1}{p} + \frac{1}{q} = 1$. This strategy  was used in  Coulhon
and Duong \cite{CD2} by looking at the heat kernel on differential
forms.  They also made an interesting connection  between
boundedness of the Riesz transform and  Littlewood-Paley-Stein
inequalities. 

Let us recall the B\"ochner formula
  $$\DDelta=\nabla^*\nabla+R^+-R^-$$
where $R^+$ and $R^-$ are respectively the positive and negative
part of the Ricci curvature and $\nabla$ denotes the Levi-Civita
connection on $M$. We use this  formula to look at  $\DDelta$ as a Schr\"odinger operator  on $1$-forms and then to  bring known  techniques
for the Riesz transforms of Schr\"odinger operators on functions and try to adapt them to this setting. We note however that  the boundedness of the Riesz transform of a Schr\"odinger operator is a delicate task even in the Euclidean setting. See  Assaad and Ouhabaz \cite{AO} and also the last sections of the present paper.

First, we assume that the negative part $R^-$ is $\alpha$-subcritical for some  $\alpha\in [0,1)$, that is
\begin{eqnarray}\label{eq:conditionR-}
0\leq (R^-\omega,\omega)\leq \alpha(H\omega,\omega) := \alpha( (\nabla^*\nabla +R^+)\omega, \omega) {\mbox{ for
all }} \omega\in C_c^\infty(\forms),
\end{eqnarray}
where $(\cdot,\cdot)$ is the inner product in $L^2(\Lambda^1T^*M)$. Then we prove that $d^* \DDelta^{-\f12}$ is
bounded from $H_{\DDelta}^1(\forms)$ to $L^1(M)$
where $H_{\DDelta}^1(\forms)$ is the Hardy spaces associated with the operator $\DDelta$,  see Section~\ref{sec:Hardy spaces} for details
and definitions.
By interpolation,  we obtain that   $d^* \DDelta^{-\f12}$ is bounded from $L^p(\forms)$ to $L^p(M)$ for all $p\in(1,2]$ if
$D\leq 2$ and all $p\in (p_0',2]$ if $D>2$ where
$p_0:=\frac{2D}{(D-2)(1-\sqrt{1-\alpha})}$. The latter result was  proved recently by Magniez  \cite{M}.  As a corollary, the Riesz transform $d \Delta^{-\f12}$ is bounded from $L^p(M)$  to $L^p(\forms)$ for all $p\in
(1,\infty)$ if $D\leq 2$ and all $p\in (1,p_0)$ if $D>2$.

 The above $L^p$-boundedness result is not sharp in general. Note that if  $\alpha$ is close to $1$ and  $D$ is large, then $p_0$ is close to $2$.
In \cite{D}, Devyver  proved that if $M$ satisfies a Sobolev inequality together with the additional assumption that  balls of large  radius have a polynomial volume growth, that is $cr^D\leq v(x,r)\leq Cr^D$ for $r\ge 1$,  and $R^-\in L^{\frac{D}{2}-\varepsilon}\cap L^{\infty}$ for some $\varepsilon>0$, then the Riesz transform $d \Delta^{-\f12}$ is bounded from $L^p(M)$  to $L^p(\forms)$ for all $p\in
(1,D)$. 

Our aim is to get a similar result under the doubling condition for the volume without assuming the  Sobolev inequality.
We suppose that the negative part of the Ricci curvature $R^-$ satisfies
\begin{eqnarray}\label{eq:vol}
\|R^-\|_{vol}:=\int_0^1 \left\|\frac{|R^-|^{\frac{1}{2}}}{v(\cdot,\ \sqrt{t})^{\frac{1}{p_1}}}\right\|_{p_1}\frac{dt}{\sqrt{t}}+\int_1^\infty \left\|\frac{|R^-|^{\frac{1}{2}}}{v(\cdot,\ \sqrt{t})^{\frac{1}{p_2}}}\right\|_{p_2}\frac{dt}{\sqrt{t}}<\infty,
\end{eqnarray}
for some $p_1, p_2>2$. We prove that if $3<p_2\le D$ then the  Riesz transform $d\Delta^{-\frac{1}{2}}$ is bounded on $L^p$ for all $1<p<p_2$. In the particular case where the volume of balls has polynomial lower bound, $v(x, r) \ge Cr^D$, then our condition $\|R^-\|_{vol}<\infty$
is satisfied if  $R^-\in L^{\frac{D}{2}-\eta}\cap L^{\frac{D}{2}+\eta} $ for some $\eta>0$. In this situation, $p_2$ can be chosen to be  any number smaller than $D$ and  our result  shows that $d\Delta^{-\frac{1}{2}}$ is bounded on $L^p$ for all $1<p< D$.  This latter result recovers and extends a result due to  Devyver \cite{D} who assumes $R^- \in L^{\frac{D}{2}-\eta}\cap L^{\infty}$. Let us also mention recent results by Carron
\cite{C2} who proved in particular that if the negative part of the Ricci curvature has at most a  quadratic growth and the volume satisfies a reverse doubling condition with a "dimension" $\nu$, then $d\Delta^{-\frac{1}{2}}$ is bounded on $L^p$ for $p \in (1, \nu)$.

In the last two sections of the paper  we consider the Riesz transform of Schr\"odinger operators. Let $A=\Delta+V$ with signed potential $V=V^+-V^-$. Similarly to our first result, we assume that  $V^+\in L_{loc}^1$ and $V^-$ satisfies  $\alpha$-subcritical condition for
some $\alpha\in [0,1)$ :
\begin{eqnarray}\label{eq:conditionV-}
\int_M V^-u^2d\mu\leq \alpha\Big[\int_M |\nabla
u|^2d\mu+\int_MV^+u^2d\mu\Big] \, {\rm for \ all }\  u\in W^{1,2}(M).
\end{eqnarray}

Under this assumption on the potential $V$, we prove that the associated Riesz
transform $d A^{-\frac{1}{2}}$ is bounded from $H_A^1(M)$ to $L^1(\forms)$. By interpolation we obtain that  $d A^{-\frac{1}{2}}$
 is bounded on $L^p(M)$ for all $p\in(1,2]$ if $D\leq 2$ and all
$p\in (p_0',2]$ if $D>2$, where again $p_0:=\frac{2D}{(D-2)(1-\sqrt{1-\alpha})}$.  The latter result is proved by  Assaad and Ouhabaz \cite{AO} by a different approach.
For $p>2$, we assume in addition that the negative part of the Ricci curvature $R^-$ satisfies~\eqref{eq:vol} and  also $V$ satisfies~\eqref{eq:vol}
that is
\begin{equation}\label{Vvol}
\int_0^1 \left\|\frac{|V|^{\frac{1}{2}}}{v(\cdot,\ \sqrt{t})^{\frac{1}{p_1}}}\right\|_{p_1}\frac{dt}{\sqrt{t}}+\int_1^\infty \left\|\frac{|V|^{\frac{1}{2}}}{v(\cdot,\ \sqrt{t})^{\frac{1}{p_2}}}\right\|_{p_2}\frac{dt}{\sqrt{t}}<\infty.
\end{equation}
 Then we  prove that $dA^{-\frac{1}{2}}$ is bounded on $L^p$ for all $p_0'<p<\frac{p_0 r}{p_0+r}$ where $r=\inf(p_1,p_2)$.
In the particular case where  the volume $v(x,r)$ has polynomial growth and $V^-=0$, our result implies that $dA^{-\frac{1}{2}}$ is bounded on $L^p$ for all $1<p<D$ provided  $V \in L^{\frac{D}{2} - \eta} \cap L^{\frac{D}{2} + \eta}$  for some $\eta > 0$. In the last section  we prove that the interval $(1, D)$ cannot be improved in general.  More precisely, we assume that the manifold $M$ satisfies the Poincar\'e inequality, the doubling condition~\eqref{def:doubling2} with the doubling constant $D$ and that there exists a positive bounded function $\phi$ such that  $A\phi=0$. We then prove that if the Riesz transform $d A^{-\frac{1}{2}}$ is bounded on $L^p$ for some $p>D$, then $V=0$. A similar result was proved by Guillarmou and Hassell \cite{GH} on complete noncompact and asymptotically conic Riemannian manifold assuming $V$ is smooth and sufficiently vanishing at infinity. In particular this is satisfied on the Euclidean space $\mathbb{R}^n$ with a smooth and compactly supported potential $V$.

\bigskip

Throughout, the symbols ``$c$" and ``$C$" will denote (possibly
different) constants that are independent of the essential
variables.

\medskip
\section{Hardy spaces associated with self-adjoint operators}
\label{sec:Hardy spaces} \setcounter{equation}{0}

In this section, we recall Hardy
spaces $H_L^p$ associated with a given  operator $L$ on manifolds. The operator $L$ is either acting on functions or  on
differential $1$-forms. These Hardy spaces have been  studied by several authors,  see for
example \cite{Au, AMR, AMM, DY2, HLMMY, HM, HMMc}.

Let $(X, \rho, \mu)$  be a metric measured space   satisfying the
doubling condition~\eqref{def:doubling2}. We denote by  $TX$ a
smooth vector bundle over  $X$ with scalar product
$(\cdot,\cdot)_x$. For $f(x)\in T_xX$ we put
$|f(x)|_x^2=(f(x),f(x))_x$. To simplify the notation we will write
$|\cdot|$ instead of $|\cdot|_x$.
 In the next sections of this paper, $X$ will be a
complete non-compact Riemannian manifold.

First, we recall the definitions of the finite speed propagation property and the Davies-Gaffney estimates for semigroup.
For $r>0$, we set
\begin{equation*}
\D_r:=\{ (x,\, y)\in X\times X: {\rho}(x,\, y) \le r \}.
\end{equation*}
Given an operator $T$ on $L^2(TX)$, we write
\begin{equation}\label{eq:include}
\supp \, K_{T} \subseteq \D_r
\end{equation}
if $\langle T f_1, f_2 \rangle = 0$ for all  $f_k \in C(TX)$ with
$\supp \, f_k \subseteq B(x_k,r_k)$ for $k = 1,2$ and
$r_1+r_2+r < {\rho}(x_1, x_2)$.
 If $T$ is an integral operator
with kernel  $K_T$,  then \eqref{eq:include} has  the usual meaning $K_T(x,y) = 0$ for a.e. $(x,y) \in X \times X$ with $\rho(x,y) > r$.

\begin{definition}\label{def:finitspeed}
 Given a  non-negative self-adjoint operator $L$  on $L^2(TX)$.
One says that the operator $L$ satisfies the {\emph {finite speed
propagation property}} if
$$
{\rm supp} \, K_{\cos(t\SL)} \subseteq \D_t  \quad {\rm for \,
all\,}\, t\ge 0\,. \leqno{\rm (FS)}
$$
\end{definition}

\begin{definition}\label{def:DavisGaffney}
One says that the semigroup $\{e^{-tL}\}_{t>0}$ generated by (minus) $L$
satisfies the {\it Davies-Gaffney estimates} if there exist constants
$C$, $c>0$ such that for all open subsets $U_1,\,U_2\subset X$ and
all $t>0$,
\begin{equation}\label{eq:DGestimates}
|\langle e^{-tL}f_1, f_2\rangle| \leq C\exp\Big(-{{\rm
dist}(U_1,U_2)^2\over c\,t}\Big) \|f_1\|_{L^2(TX)}\|f_2\|_{L^2(TX)},
\nonumber \leqno{\rm (DG)}
\end{equation}
for every $f_i\in L^2(TX)$ with $\mbox{supp}\,f_i\subset U_i$,
$i=1,2$, where ${\rm dist}(U_1,U_2):=\inf_{x\in U_1, y\in U_2}
\rho(x,y)$.
\end{definition}

The following result is taken from  \cite[Theorem~2]{S}.

 \begin{proposition}\label{prop2.1} Let $L$ be a non-negative self-adjoint operator
acting on $L^2(TX)$. Then the finite speed propagation property {\rm
(FS)} and Davies-Gaffney estimates {\rm (DG)} are equivalent.
\end{proposition}

\medskip

Next we recall the definition of Hardy spaces associated with self-adjoint operators.
Assume that the operator $L$ satisfies the Davies-Gaffney estimates~{\rm
(DG)}. Following \cite{AMR, AMM, DY2, HLMMY} one can define the $L^2$ adapted Hardy
space by
\begin{equation}\label{eq2.H2}
H^2(TX) := \overline{R(L)},
\end{equation}
that is, the closure of the range of $L$ in $L^2(TX)$.  Then $L^2(TX)$
is the orthogonal sum of $H^2(TX)$ and the null space $N(L)$.
Consider the following quadratic functional associated to $L$:
\begin{eqnarray}
S_{K}f(x):=\Big(\int_0^{\infty}\!\!\!\!\int_{\substack{
\rho(x,y)<t}} |(t^2L)^{K}e^{-t^2L} f(y)|^2 {d\mu(y)\over
v(x,t)}{dt\over t}\Big)^{\frac{1}{2}}, \label{e2.3}
\end{eqnarray}
 where $x\in X$, $f\in L^2(TX)$ and $K$ is a natural number. For each $K\geq 1$ and $0<
p<\infty$, we now define
\begin{eqnarray*}
D_{K, p} :=\Big\{ f\in H^2(TX): \ S_{K}f\in L^p(X)\Big\}.
\end{eqnarray*}

\begin{definition}{\label{def2.2}} Let $L$ be
 a  non-negative self-adjoint operator on $L^2({TX})$
 satisfying the Davies-Gaffney estimates {\rm (DG)}. \

(i) For each $p\in (0,2]$, the Hardy space $H^p_{L}(TX)$ associated
with $L$  is the completion of the space $D_{1, p}$ with respect to  the norm
$$
 \|f\|_{H_{L}^p(TX)}:=  \|S_{1}f\|_{L^p(X)}.
$$

(ii) For each $p\in (2,\infty)$, the Hardy space $H^p_{L}(TX)$
associated with $L$ is the completion of the space $D_{K_0, p}$ in the
norm
$$
\|f\|_{H_{L}^p(TX)}:=  \|S_{K_0}f\|_{L^p(X)}, \ \ \ {\rm where}\,\,
K_0=\Big[\,{D\over 4}\,\Big]+1.
$$
\end{definition}

It can be verified that the dual of $H^p_{L}(TX)$ is
$H^{p'}_{L}(TX)$, with $\frac{1}{p}+ \frac{1}{p}'=1$ (see Proposition 9.4 of
\cite{HLMMY}). We also have complex interpolation  and
Marcinkiewicz-type interpolation results between $H^p_L(TX)$ for
$1\leq p\leq 2$ (see Proposition 9.5 and Theorem 9.7
\cite{HLMMY}). Although the above results in \cite{HLMMY} are stated and proved on $X$ and not on $TX$, the whole machinery
developed there works  in the context of Hardy spaces over  $TX$.

  Note that if we only assume
Davies-Gaffney estimates on the heat kernel of $L$, for
$1<p<\infty$, $p\not= 2$, $H^p_{L}(TX)$ may or may not coincide with
the space $L^p(TX)$. For the relation of $H_L^p(TX)$ and $L^p(TX)$ for $1<p\leq 2$, we have
the following proposition.

\begin{proposition}\label{prop:HpCoinsidesLp}
    Let  $L$ be an injective, nonnegative self-adjoint operator
    on $L^2(TX)$ satisfying the finite propagation speed
    property ${\rm (FS)}$ and general $(p_0,2)$-Davies-Gaffney estimates
$$
\|\chi_{B(x,t)}e^{-t^2L}\chi_{B(y,t)}\|_{p_0\to 2}\leq
Cv(x,t)^{\frac{1}{2}-\frac{1}{p}}\exp\Big(-c\rho(x,y)^2/t^2\Big)
\leqno{\rm (DG_{p_0})}
$$
    for some $p_0$ with $1\leq p_0 \leq 2$. Then for each $p$ with
    $p_0 < p \leq 2$, the Hardy space $H^{p}_{L}(TX)$ and the Lebesgue space $L^p(TX)$ coincide and their norms are equivalent.
\end{proposition}

\begin{proof}
This proof is the same as  for  Hardy space $H^p_L(X)$ (see \cite[Proposition
9.1(v)]{HMMc} and~\cite{Au}). For more  details, see also
\cite[Theorem 4.19]{U}.
\end{proof}

We denote
by ${\mathcal D}(L)$ the domain of the  operator $L$.
 The following definition of atoms of Hardy spaces associated with operators  was
introduced in \cite{HLMMY}.

\begin{definition}{\label{def:atom}}
Let $M$ be a positive integer. A function $a\in L^2(TX)$ is called a $(1,2,M)$-atom associated with
$L$ if there exist a function $b\in {\mathcal D}(L^M)$ and a ball $B $
such that

\smallskip

{\rm  (i)}\ $a=L^Mb$,

\smallskip

{\rm  (ii)}\ {\rm supp}\  $L^{k}b\subset B, \ k=0, 1,\ldots,M$,

\smallskip

{\rm (iii)}\ $||(r_B^2L)^{k}b||_{L^2}\leq
r_B^{2M}v(B)^{-\frac{1}{2}},\ k=0,1,\ldots,M$, where $v(B)$ is the
volume of the ball $B$.
\end{definition}

\medskip
We can define the atomic Hardy space $H_{L,at,M}^1(TX)$ as follows.
First, we say that  $f=\sum \lambda_ja_j$ is an atomic $(1,
2,M)$-representation if $\{\lambda_j\}_{j=0}^\infty\in \ell^1$, each
$a_j$ is a $(1,2,M)$-atom and the sum converges in $L^2(TX)$. Then
set
$$
\mathbb{H}_{L,at,M}^1(TX):=\{f:\,f \,\,{\mbox{has an atomic $(1,
2,M)$-representation}}\},
$$
with the norm given by
$$
\|f\|_{\mathbb{H}_{L,at,M}^1(TX)}:=\inf\{\sum_{j=0}^\infty
|\lambda_j|: f=\sum_{j=0}^\infty \lambda_ja_j \,\,{\mbox{is an
atomic $(1, 2,M)$-representation}}\}.
$$
The space $H_{L,at,M}^1(TX)$ is then defined as the completion of
$\mathbb{H}_{L,at,M}^1(TX)$ with respect to this norm. According to
\cite[Theorem 2.5]{HLMMY}, if $M>D/4$ and $L$ satisfies
Davies-Gaffney estimate ${\rm (DG)}$, then Hardy space $H_L^1(TX)$
coincides with the atomic Hardy space $H_{L,at,M}^1(TX)$ and their
norms are equivalent.

\section{Boundedness of Riesz transforms on Hardy spaces of forms}
\label{sec:H1toL1} \setcounter{equation}{0}

To prove the boundedness of the Riesz transform on Hardy spaces
associated with self-adjoint operators on forms, we need the following lemma.
\begin{lemma}\label{le:H1toL1}
Assume that $T$ is a non-negative sublinear
operator and bounded from $L^2(TX)$ to $L^2(X)$. Also assume that for every
$(1,2,M)$-atom $a$, we have
$$
\|Ta\|_{L^1(X)}\leq C
$$
with constant $C$ independent of $a$. Then $T$ is bounded from
$H^1_L(TX)$ to $L^1(X)$.
\end{lemma}
\begin{proof}
For the proof, we refer the reader to \cite[Lemma 4.3 and Proposition 4.13]{HLMMY}.
\end{proof}

 We now state  a criterion that allows  to derive  estimates on
Hardy spaces $H^p_L(TX)$. It is already stated in  \cite[Theorem
3.1]{DY} for spectral multipliers  $T = m(L)$ on  $H^p_L(X)$. We show that the arguments there
are valid for a general  linear operator $T$ which is
bounded on $L^2$. We do not require the  commutation of  $T$ with  the semigroup $e^{-tL}$.

 Let $U_j(B)=2^{j+1}B\setminus 2^jB=U_j$ when $j\geq2$ and $U_1(B)=4B$.
\begin{lemma}\label{le:H1toL1Key} Let $L$ be  a  non-negative self-adjoint operator acting on $L^2({TX})$
and satisfying the Davies-Gaffney estimates {\rm (DG)}.
 Let $T$ be a linear operator which is  bounded from $L^2(TX)$ to $L^2(X)$.
 Assume that there exist constants $M\geq 1$,  $ s>D/2$  and $C>0$
  such that for every  $j=1,2,\ldots,$
  \begin{eqnarray}\label{eq:SingularIteCondition}
\big\|T(I-e^{-r^2L})^M f\big\|_{L^2(U_j(B))}\leq C 2^{-js}
\|f\|_{L^2(B)}
\end{eqnarray}
 for every ball $B$ with    radius $r$ and for all $f\in L^2(TX)$ with supp $f\subset B$.
 Then the operator $T$ extends to a bounded operator from $H^1_L(TX)$
 to $L^1(X)$.
 \end{lemma}
\begin{proof}
Let $a=L^Mb$ be a $(1,2,M)$-atom. By Lemma~\ref{le:H1toL1}, it is enough to prove that
\begin{eqnarray}\label{eq:atomL1}
\|Ta\|_{L^1(X)}\leq C
\end{eqnarray}
with constant $C$ independent of the atom $a$.

Denote $B:=B(x,r)$ the ball containing the support of the atom $a$.
We have
$$
\|Ta\|_{L^1(X)}\leq \sum_{j=1}^{\infty}\|Ta\|_{L^1(U_j(B))}.
$$
Note that by H\"older's inequality and $L^2$-boundedness of the
operator $T$
\begin{eqnarray*}
\|Ta\|_{L^1(4B)}\leq v(4B)^{\frac{1}{2}}\|Ta\|_{L^2(X)}\leq
Cv(B)^{\frac{1}{2}}\|a\|_{L^2(TX)}.
\end{eqnarray*}
By (iii) of Definition~\ref{def:atom}, $\|a\|_{L^2}\leq
v(B)^{-\frac{1}{2}}$ and thus
\begin{eqnarray}\label{eq:InBall}
\|Ta\|_{L^1(4B)}\leq Cv(B)^{\frac{1}{2}}v(B)^{-\frac{1}{2}}\leq C.
\end{eqnarray}

Then we only need to prove that there exist some constants
$\varepsilon>0$ and $C>0$ independent of the atom $a$ such that
\begin{eqnarray}\label{eq:OutBall}
\|Ta\|_{L^1(U_j(B))}\leq C2^{-j\varepsilon}
\end{eqnarray}
for $j=2,3,\ldots$

Following (8.7) and (8.8) in \cite{HM} or (3.5) in \cite{DY},
we write
\begin{eqnarray*}
I&=&2(r^{-2}\int_r^{\sqrt{2}r}tdt)\cdot I\\
&=&2r^{-2}\int_r^{\sqrt{2}r}t(I-e^{-t^2L})^Mdt
+\sum_{\alpha=1}^MC_{j,M}r^{-2}\int_r^{\sqrt{2}r}te^{-jt^2L}dt,
\end{eqnarray*}
where $C_{\alpha,M}$ are some constants depending only on $\alpha$
and $M$ only.  Using the fact that  $\partial_te^{-\alpha
t^2L}=-2\alpha tLe^{-\alpha t^2L}$ and  applying the procedure $M$ times, we
have for every function $f$ on $TX$,
\begin{eqnarray}\label{eq:decomf}
f&=&2^M\Big(r^{-2}\int_r^{\sqrt{2}r}t(I-e^{-t^2L})^Mdt\Big)^Mf\nonumber\\
&+&\sum_{\beta=1}^M r^{-2\beta}(I-e^{-r^2L})^\beta
\Big(r^{-2}\int_r^{\sqrt{2}r}t(I-e^{-t^2L})^Mdt\Big)^{M-\beta}
\sum_{\alpha=1}^{(2M-1)\beta}C_{\beta,\alpha,M}e^{-\alpha r^2L}L^{-\beta}f\nonumber\\
&:=&\sum_{\beta=0}^{M-1}r^{-2\beta}r^{-2}\int_r^{\sqrt{2}r}t
F_{\beta,M,r}(L)(I-e^{-t^2L})^ML^{-\beta}f dt\nonumber\\
&&+r^{-2M}F_{M,M,r}(L)(I-e^{-r^2L})^ML^{-M}f
\end{eqnarray}
where
$$
F_{\beta,M,r}(L)=(I-e^{-r^2L})^\beta
\Big(r^{-2}\int_r^{\sqrt{2}r}t(I-e^{-t^2L})^Mdt\Big)^{M-\beta-1}
\sum_{\alpha=1}^{(2M-1)\beta}C_{\beta,\alpha,M}e^{-\alpha r^2L}
$$
for $0\leq \beta\leq M-1$ and
$$
F_{M,M,r}(L)= \sum_{\alpha=1}^{(2M-1)M}C_{M,\alpha,M}e^{-\alpha
r^2L}.
$$
It follows from the Davies-Gaffney estimates that the operator
$F_{\beta,M,r}(L)$, $\beta=0,1,\cdots,M$, satisfies $L^2$
off-diagonal estimates, which means that there exist some constants $c,C>0$ such
that
\begin{eqnarray}\label{eq:offdia}
\|F_{\beta,M,r}(L)f\|_{L^2(U_j(B))}\leq C
e^{-c4^{|j-i|}}\|f\|_{L^2(U_i(B))}.
\end{eqnarray}
For the details, see \cite[pp. 307-309]{DY}.

By (i) of Definition~\ref{def:atom}, $a=L^Mb$. Then applying \eqref{eq:decomf}, we have
\begin{eqnarray*}
Ta&=&\sum_{\beta=0}^{M-1}r^{-2\beta}r^{-2}\int_r^{\sqrt{2}r}tT(I-e^{-t^2L})^MF_{\beta,M,r}(L)L^{M-\beta}b dt\nonumber\\
&&+r^{-2M}T(I-e^{-r^2L})^MF_{M,M,r}(L)b.
\end{eqnarray*}

Then by H\"older's inequality
\begin{eqnarray}\label{eq:L1toL2}
\lefteqn{\|Ta\|_{L^1(U_j(B))}}\nonumber\\
&\leq& C v(U_j(B))^{\frac{1}{2}}\sum_{\beta=0}^M
r^{-2\beta}\sup_{t\in[r,\sqrt 2 r]}
\|T(I-e^{-t^2L})^MF_{\beta,M,r}(L)L^{M-\beta}b\|_{L^2(U_j(B))}.
\end{eqnarray}
For $t\in [r,\sqrt2 r]$, let $B^t:=B(x,t)$.
Then
\begin{eqnarray}\label{eq:iDecomposition}
\lefteqn{\|T(I-e^{-t^2L})^MF_{\beta,M,r}(L)L^{M-\beta}b\|_{L^2(U_j(B^t))}}\nonumber\\
&\leq &
\sum_{i=1}^\infty\|T(I-e^{-t^2L})^M\chi_{U_i(B^t)}F_{\beta,M,r}(L)L^{M-\beta}b\|_{L^2(U_j(B^t))}.
\end{eqnarray}

For $|i-j|\leq 4$, by $L^2$ boundedness of $T(I-e^{-t^2L})^M$ and off-diagonal estimates \eqref{eq:offdia},
\begin{eqnarray}\label{eq:ijClose}
\lefteqn{ \|T(I-e^{-t^2L})^M\chi_{U_i(B^t)}F_{\beta,M,r}(L)L^{M-\beta}b\|_{L^2(U_j(B^t))} }\nonumber\\
&\leq & \|F_{\beta,M,r}(L)L^{M-\beta}b\|_{L^2(U_i(B^t))}\nonumber\\
&\leq &C e^{-c4^i}\|L^{M-\beta}b\|_{L^2}.
\end{eqnarray}

For $i\leq j-4$, we decompose $U_i(B^t)$ as the union of a  finite
number of balls $B_{\kappa,i}=B(x_{\kappa,i},t)$, the number is
compared with $2^{iD}$ and $\mbox{dist}\,(B_{\kappa,i},B)\geq
C2^ir$. For each $B_{\kappa,i}$, we can write
$$
U_j(B^t)\subset\bigcup_{\ell=0}^{2i}U_{j-i+\ell}(B_{\kappa,i}).
$$
Thus by condition~\eqref{eq:SingularIteCondition} and off-diagonal estimates \eqref{eq:offdia},
\begin{eqnarray*}
\lefteqn{ \|T(I-e^{-t^2L})^M\chi_{U_i(B^t)}F_{\beta,M,r}(L)L^{M-\beta}b\|_{L^2(U_j(B^t))}}\\
&\leq& \sum_\kappa \sum_\ell
\|T(I-e^{-t^2L})^M\chi_{B_{\kappa,i}}F_{\beta,M,r}(L)L^{M-\beta}b\|_{L^2(U_{j-i+\ell}(B_{\kappa,i}))}\\
&\leq& C\sum_\kappa \sum_\ell
2^{-(j-i+\ell)s}\|F_{\beta,M,r}(L)L^{M-\beta}b\|_{L^2(B_{\kappa,i})}\\
&\leq& C\sum_\kappa \sum_\ell 2^{-(j-i+\ell)s} e^{-c4^i}\|L^{M-\beta}b\|_{L^2}\\
&\leq& C2^{iD} 2^{-(j-i)s}e^{-c4^i}\|L^{M-\beta}b\|_{L^2}.
\end{eqnarray*}
Thus
\begin{eqnarray}\label{eq:ileqj}
\lefteqn{ \sum_{i=1}^{j-4}\|T(I-e^{-t^2L})^M\chi_{U_i(B^t)}F_{\beta,M,r}(L)L^{M-\beta}b\|_{L^2(U_j(B^t))} }\nonumber\\
&\leq& C 2^{-js}\|L^{M-\beta}b\|_{L^2}\sum_{i=1}^{j-4} 2^{iD+is}
e^{-c4^i}\leq C 2^{-js}\|L^{M-\beta}b\|_{L^2}.
\end{eqnarray}

For $i\geq j+4$, decompose $U_i(B^t)$ as the union of finite number
of balls $B_{\kappa,i}=B(x_{\kappa,i},t)$, the number is compared
with $2^{iD}$ and $\mbox{dist}\,(B_{\kappa,i},B)\geq C2^ir$. For any
$B_{\kappa,i}$, we can write
$$
U_j(B^t)\subset\bigcup_{\ell=-2}^{2j+1}U_{i-j+\ell}(B_{\kappa,i}).
$$
Thus by condition~\eqref{eq:SingularIteCondition} and off-diagonal estimates \eqref{eq:offdia},
\begin{eqnarray*}
\lefteqn{\|T(I-e^{-t^2L})^M\chi_{U_i(B^t)}F_{\beta,M,r}(L)L^{M-\beta}b\|_{L^2(U_j(B^t))}}\\
&\leq& \sum_\kappa \sum_\ell
\|T(I-e^{-t^2L})^M\chi_{B_{\kappa,i}}F_{\beta,M,r}(L)L^{M-\beta}b\|_{L^2(U_{i-j+\ell}(B_{\kappa,i}))}\\
&\leq& C\sum_\kappa \sum_\ell
2^{-(i-j+\ell)s}\|F_{\beta,M,r}(L)L^{M-\beta}b\|_{L^2(B_{\kappa,i})}\\
&\leq& C\sum_\kappa \sum_\ell 2^{-(i-j+\ell)s} e^{-c4^i}\|L^{M-\beta}b\|_{L^2}\\
&\leq& C2^{iD} 2^{-(i-j)s}e^{-c4^i}\|L^{M-\beta}b\|_{L^2}.
\end{eqnarray*}
Thus
\begin{eqnarray}\label{eq:igeqj}
\lefteqn{\sum_{i=j+4}^{\infty}\|T(I-e^{-t^2L})^M\chi_{U_i(B^t)}F_{\beta,M,r}(L)L^{M-\beta}b\|_{L^2(U_j(B^t))}
}\nonumber\\
&\leq& C 2^{js}\|L^{M-\beta}b\|_{L^2}\sum_{i=j+4}^{\infty} 2^{iD-is}
e^{-c4^i}\leq C 2^{-js}\|L^{M-\beta}b\|_{L^2}.
\end{eqnarray}

Combining the estimates \eqref{eq:iDecomposition}, \eqref{eq:ijClose}, \eqref{eq:ileqj} and \eqref{eq:igeqj}, it follows that
$$
\|T(I-e^{-t^2L})^MF_{\beta,M,r}(L)L^{M-\beta}b\|_{L^2(U_j(B^t))}\leq
C 2^{-js}\|L^{M-\beta}b\|_{L^2}.
$$
Noting that for $t\in [r,\sqrt2 r]$, we have
$$
U_j(B)\subset U_j( B^t)\cup U_{j-1}( B^t).
$$
Thus, by \eqref{eq:L1toL2} and by (iii) of Defintion~\ref{def:atom},
\begin{eqnarray*}
\|Ta\|_{L^1(U_j(B))}
&\leq & C
v(U_j(B))^{\frac{1}{2}}r^{-2\beta}2^{-js}\|L^{M-\beta}b\|_{L^2}\\
&\leq & C2^{-js}v(U_j(B))^{\frac{1}{2}}r^{-2\beta}r^{2\beta}v(B)^{-\frac{1}{2}}\\
&\leq & C2^{-j(s-D/2)},
\end{eqnarray*}
which proves \eqref{eq:OutBall}. Then combining estimate~\eqref{eq:InBall}, we
complete the proof of Lemma~\ref{le:H1toL1Key}.
\end{proof}

\medskip

Now we state the  main result  of  this section.

\begin{theorem}\label{th:Rp<2}
Let $M$ be a complete non-compact Riemannian manifold satisfying
assumptions \eqref{def:doubling2} with doubling dimension $D$, Gaussian upper bound~\eqref{def:Gassianbound} and $\alpha$-subcritical condition~\eqref{eq:conditionR-}. Then the
associated Riesz transform $d^* \DDelta^{-\f12}$ is

i) bounded from $H_{\DDelta}^1(\forms)$ to $L^1(M)$,

ii) bounded from $H_{\DDelta}^p(\forms)$ to $L^p(M)$ for all $p\in
[1,2]$,

iii)  bounded from $L^p(\forms)$ to $L^p(M)$ for all $p\in(1,2]$ if
$D\leq 2$ and all $p\in (p_0',2]$ if $D>2$ where
$p_0:=\frac{2D}{(D-2)(1-\sqrt{1-\alpha})}$.
\end{theorem}


\begin{proof}
We apply  Lemma~\ref{le:H1toL1Key}. Let $X=M$, $TX=\forms$ and $L=\DDelta$.

The estimate~\eqref{eq:SingularIteCondition} was
proved in the proof of Theorem~1.1 in~\cite{M} (the proof of estimate (34) in Page 23). This gives assertion i). 

Assertion ii) follows from i) by interpolation and the fact that  $d^* \DDelta^{-\f12}$ is bounded from $L^2(\forms)$ to $L^2(M)$.

Finally, the Davies-Gaffney estimate~$(DG_{p})$ was proved  in~\cite{M}, Theorem 4.1.  We then apply Proposition~\ref{prop:HpCoinsidesLp}
to obtain  iii).
\end{proof}

By duality and the commutation formula
\begin{equation}\label{commit}
 \DDelta d = d \Delta
 \end{equation}
 we obtain  the following corollary of
Theorem~\ref{th:Rp<2}.
\begin{corollary}\label{cor:Rp>2}
Let $M$ be a complete non-compact Riemannian manifold satisfying
assumptions \eqref{def:doubling2} with doubling dimension $D$, Gaussian upper bound~\eqref{def:Gassianbound} and $\alpha$-subcritical condition~\eqref{eq:conditionR-}. Then the
associated Riesz transform $d \Delta^{-\f12}$ is

i) bounded from $L^p(M)$ to $H_{\DDelta}^p(\forms)$ for all $p\in
[2,\infty)$,

ii) bounded from $L^p(M)$  to $L^p(\forms)$ for all $p\in
[2,\infty)$ if $D\leq 2$ and all $p\in [2,p_0)$ if $D>2$ where
$p_0:=\frac{2D}{(D-2)(1-\sqrt{1-\alpha})}$.
\end{corollary}

As mentioned in the introduction, assertion ii) was already proved in \cite{M}.


\section{The Riesz transform for $p>2$}
\label{sec:pbigerthan2} \setcounter{equation}{0}

Our aim in this section is to investigate  boundedness of the Riesz transform $d\Delta^{-\frac{1}{2}}$ on $L^p$ for other values of $p>2$  that are not covered by the previous corollary. In order to do this we make an integrability  assumption  on  the Ricci curvature.

Our main result in this section is the following theorem.

\begin{theorem}\label{th:LPp>2}
Assume that the Riemannian manifold $M$ satisfies the doubling condition~\eqref{def:doubling2} and the
 Gaussian upper bound ~\eqref{def:Gassianbound}. Assume that the negative part $R^-$  of the Ricci curvature $R$ satisfies~\eqref{eq:vol}
 for some $p_1$ and $p_2$ such that  $3<p_2\leq D$. Then the Riesz transform $d\Delta^{-\frac{1}{2}}$ is bounded from  $L^p(M) $  to
 $L^p(\forms)$ for $1<p<p_2$.
\end{theorem}

\begin{remark}
Suppose that $v(x,r) \ge Cv(r)$ for all $r > 0$. Then \eqref{eq:vol} is satisfied for $p_1$ and $p_2$ such that
$$\int_0^1 v(t)^{-1/p_1} dt < \infty\;\text{and}\;R^-\in L^{\frac{p_1}{2}}$$
and
$$\int_1^\infty v(t)^{-1/p_2} dt < \infty\;\text{and}\;R^-\in L^{\frac{p_2}{2}}.$$
In the particular case where $v(r) = r^D$ for all $r > 0$, \eqref{eq:vol} is then satisfied if  $R^-\in L^{\frac{D}{2}-\eta}\cap L^{\frac{D}{2}+\eta} $ for some $\eta>0$. Therefore, the Riesz transform  $d\Delta^{-\frac{1}{2}}$ is bounded from $L^p(M) $  to
 $L^p(\forms)$ for all $p$ with $1<p<D$. We recover and extend a result of Devyver \cite{D} who assumed
 $R^-\in L^{\frac{D}{2}-\eta}\cap L^{\infty}$.
 \end{remark}

Before we start the proof of the theorem we state  the following result on $L^p-L^q$ estimates for perturbations of $\DDelta$ by a
non-negative potential. The manifold $M$ satisfies the same assumptions as in the previous theorem.
\begin{theorem}\label{JoV}
Let $\mathcal{R}=\mathcal{R}^+-\mathcal{R}^-$ be a field of symmetric endomorphisms acting on $\Lambda^1T^*M$. Let $\alpha\in[0,1)$. We suppose that $\mathcal{R}^-$ is $\alpha$-subcritical for the operator $\nabla^*\nabla+\mathcal{R}^+$, that is for all $\omega\in \mathcal{C}_c^{\infty}(\Lambda^1T^*M)$ 
$$(\mathcal{R}^- \omega,\omega)\le\alpha ((\nabla^*\nabla+\mathcal{R}^+)\omega,\omega).$$
 Then for every open subsets $E$ and $F$ of $M$\\
i) $\| \chi_E d^*e^{-t(\nabla^*\nabla+\mathcal{R}^+-\mathcal{R}^-)} \chi_F \|_{2\to2} \le \frac{C}{\sqrt{t}} e^{-c \frac{dist(E,F)^2}{t}}$\\
ii) $\| e^{-t(\nabla^*\nabla+\mathcal{R}^+-\mathcal{R}^-)}  \chi_{B(x,r)} \|_{p\to q} \le \frac{C}{v(x, r)^{\frac{1}{p}-\frac{1}{q}}} \max(\frac{r}{\sqrt{t}}, \frac{\sqrt{t}}{r})^\beta,$\\
where $C, c $ and $\beta$ are positive constants. The assertion  ii) holds for all  $p \le q  \in
(1,\infty)$ if $D\leq 2$ and all $p\in (p_0',p_0), q \in [p, p_0)$ if $D>2$ where
$p_0:=\frac{2D}{(D-2)(1-\sqrt{1-\alpha})}$.
\end{theorem}
If $\mathcal{R}$ is the Ricci curvature, this theorem was proved in \cite{M}, Theorem 4.1 and Corollary 4.5. In the general case, the proof is the same as in \cite{M}.

\vspace{.3cm}

\begin{proof}[Proof of Theorem \ref{th:LPp>2}]  We shall proceed in three main steps.

\vspace{.2cm}

\indent{\bf Step I}. For a fixed point $x_0\in M$, we prove that there exist a positive number $r_0$ sufficiently large and a positive function $W\in C_c^\infty(M)$ with $\support W\subset B(x_0,r_0)$ such that  the following $L^p-L^q$ estimates hold for all $p\in(p_0',p_0)$, $q\in[p,p_0)$ and $t>1$
$$\|e^{-t(\DDelta+W)}\chi_{B(x_0,r_0)}\|_{p\to q}\leq
C_{x_0,r_0}t^{-p_2(\frac{1}{p}-\frac{1}{q})/2}.$$

For any $\varepsilon>0$, by the Dominated Convergence Theorem,  we can find a large enough $r_0$ such that
\begin{eqnarray*}
 \int_0^1 \left\|\frac{|R^-|^{\frac{1}{2}}}{v(\cdot,\ \sqrt{t})^{\frac{1}{p_1}}}\right\|_{L^{p_1}(B(x_0,\frac{r_0}{2})^c)}\frac{dt}{\sqrt{t}}+\int_1^\infty \left\|\frac{|R^-|^{\frac{1}{2}}}{v(\cdot,\ \sqrt{t})^{\frac{1}{p_2}}}\right\|_{L^{p_2}(B(x_0,\frac{r_0}{2})^c)}\frac{dt}{\sqrt{t}}<\varepsilon.
\end{eqnarray*}
We construct a function $0\leq \varphi\in C_c^\infty(M)$ such that
$\varphi=1$ in the ball $B(x_0,r_0/2)$,  $\varphi\leq 1$ in $B(x_0,r_0)\backslash B(x_0,r_0/2)$ and $\varphi=0$ outside the ball $B(x_0,r_0)$. Let $V=\varphi |R^-|$. Then $V$ is a compactly supported function and smooth except when
$|R^-|=0$. We  choose another function $W$ such that $W\geq V$, $W\in C_c^\infty(M)$ and $\support W\subset B(x_0, r_0)$.
We write
 $$\DDelta+W=\nabla^*\nabla+R^++(W-R^-)^+-(W-R^-)^-.$$
Note that $(W-R^-)^-=(W-V+V-R^-)^- \le (V-R^-)^-$ and $|(V-R^-)^-|=0$ in the ball $B(x_0,r_0/2)$ and
$|(V-R^-)^-|\leq |R^-|$ outside the ball $B(x_0,r_0/2)$. Therefore,   $\|(W-R^-)^-\|_{vol}<\varepsilon $.
 From Proposition 5.8 in \cite{M}, with $(W-R^-)^-$ and $\nabla^*\nabla+R^++(W-R^-)^+$ instead of
 $R^-$ and $\nabla^*\nabla+R^+$, we have
$$
((W-R^-)^-\omega,\omega)\leq C\|(W-R^-)^-\|_{vol}^2((\nabla^*\nabla+R^++(W-R^-)^+)\omega,\omega).
$$
Then we choose  $\varepsilon$   small enough to have $C\|(W-R^-)^-\|_{vol}<C\varepsilon<1$. Therefore,
 $(W-R^-)^-$ is $\varepsilon$-subcritical with respect to  $\nabla^*\nabla+R^++(W-R^-)^+$, i.e.
 \begin{equation}\label{crit}
((W-R^-)^-\omega,\omega)\leq  \varepsilon ( (\nabla^*\nabla+R^++(W-R^-)^+)\omega, \omega)  {\mbox{ for
all }} \omega\in C_c^\infty(\forms).
\end{equation}

By  Theorem \ref{JoV}, we obtain  $L^p-L^q$ off-diagonal estimates for  $e^{-t(\DDelta+W)}$
for all $p\in (p_0',p_0), q \in [p, p_0)$ with $p_0:=\frac{2D}{(D-2)(1-\sqrt{1-\varepsilon})}$.  
Hence  for  $p\in (p_0',p_0), q \in [p, p_0)$ and  $\sqrt{t}\geq r_0$
\begin{eqnarray}\label{eq:pto2bigt}
\|e^{-t(\DDelta+W)}\chi_{B(x_0,r_0)}\|_{p\to q}\leq
\|e^{-t(\DDelta+W)}\chi_{B(x_0,\sqrt{t})}\|_{p\to q}\leq
Cv(x_0,\sqrt{t})^{-(\frac{1}{p}-\frac{1}{q})},
\end{eqnarray}
and for $1<\sqrt{t}\leq r_0$,
\begin{eqnarray}\label{eq:pto2smallt}
\|e^{-t(\DDelta+W)}\chi_{B(x_0,r_0)}\|_{p\to q}\leq
\frac{C}{v(x_0,r_0)^{\frac{1}{p}-\frac{1}{q}}}(\frac{r_0}{\sqrt{t}})^{\beta}\leq
C_{x_0,r_0}t^{-p_2(\frac{1}{p}-\frac{1}{q})/2}.
\end{eqnarray}

Note that $|V|\leq |R^-|$ and so $\|V\|_{vol}\leq \|R^-\|_{vol}\leq
C$, which gives
$$
\int_1^\infty \left\| \frac{V^{\frac{1}{2}}}{v(\cdot,
\sqrt{t})^{\frac{1}{p_2}}} \right\|_{p_2}\frac{dt}{\sqrt{t}}\leq C.
$$
Since $\support V\subset B(x_0,r_0)$, we have for $x\in B(x_0,r_0)$
and $\sqrt{t}>r_0$
$$
v(x,\sqrt{t})=\frac{v(x,\sqrt{t})}{v(x_0,\sqrt{t})}v(x_0,\sqrt{t})\leq
C \left(1+\frac{\rho(x,x_0)}{\sqrt{t}}\right)^D v(x_0,\sqrt{t})\leq
C'v(x_0,\sqrt{t}).
$$
Thus
$$
\| V^{\frac{1}{2}} \|_{p_2}\int_{r_0^2}^\infty \frac{1}{v(x_0,
\sqrt{t})^{\frac{1}{p_2}}} \frac{dt}{\sqrt{t}}\leq
C'\int_{r_0^2}^\infty \left\| \frac{V^{\frac{1}{2}}}{v(\cdot,
\sqrt{t})^{\frac{1}{p_2}}}\right\|_{p_2}\frac{dt}{\sqrt{t}} \leq C.
$$
Note that if $V=0$, then $R^-=0$ in the ball $B(x_0,r_0/2)$ and so $\|R^-\|_{vol}<\varepsilon$. As a consequence, $R^-$ satisfies the
$\varepsilon$-subcritical condition~\eqref{eq:conditionR-}.
By Corollary 1.2 in \cite{M} or our Corollary~\ref{cor:Rp>2},
$d\Delta^{-\frac{1}{2}}$ is bounded on $L^p$ for $p\in (1,p_0)$ where
$p_0:=\frac{2D}{(D-2)(1-\sqrt{1-\varepsilon})}$. With our choice of $\varepsilon$, we have $d\Delta^{-\frac{1}{2}}$ bounded on $L^p$ for
$p\in (1,p_2)$\footnote{if  $R^- = 0$, then  the Ricci curvature is non-negative and it is well known that  the Riesz transform is bounded on $L^p$ for all $p \in (1, \infty)$}. In the sequel, we assume
that $V\neq 0$. Hence
$$
\int_{r_0}^\infty \frac{1}{v(x_0, s)^{\frac{1}{p_2}}}
ds\leq\int_{r_0^2}^\infty \frac{1}{v(x_0, \sqrt{t})^{\frac{1}{p_2}}}
\frac{dt}{\sqrt{t}}\leq C\| V^{\frac{1}{2}} \|^{-1}_{p_2}\leq C'.
$$
Let $f(s)=v(x_0, s)^{-\frac{1}{p_2}}$. Note that $f$ is a positive continuous decreasing
function. So by the first mean value theorem, there exists $\xi\in [r_0,t]$ such that
\begin{eqnarray*}
\int_{r_0}^t f(s)ds= f(\xi)(t-r_0)\geq f(t)t-f(r_0)r_0.
\end{eqnarray*}
We deduce that for $t>r_0$
\begin{eqnarray*}
0<tf(t)\leq r_0f(r_0)+\int_{r_0}^t f(s)ds\leq
r_0f(r_0)+\int_{r_0}^\infty f(s)ds\leq C.
\end{eqnarray*}
That is for $t>r_0$
$$
f(t)\leq C/t
$$
and
\begin{eqnarray}\label{eq:volumelowerbound}
v(x_0,t)\geq Ct^{p_2}.
\end{eqnarray}
Therefore by \eqref{eq:pto2bigt},  \eqref{eq:pto2smallt} and \eqref{eq:volumelowerbound}, we have the
following estimate  for $p\in (p_0',p_0), q \in [p, p_0)$ and all $t>1$
\begin{eqnarray}\label{eq:pto2tbiger1}
\|e^{-t(\DDelta+W)}\chi_{B(x_0,r_0)}\|_{p\to q}\leq
C_{x_0,r_0}t^{-p_2(\frac{1}{p}-\frac{1}{q})/2}.
\end{eqnarray}

\vspace{.2cm}

\indent{\bf Step II}. We prove that the operator $d(\Delta+W)^{-\frac{1}{2}}$ is bounded on $L^p$ for all $p\in [2,p_2)$.\\
 In order to do this we take   the difference with
$(\DDelta+W)^{-\frac{1}{2}}d$, that is,
$$
d(\Delta+W)^{-\frac{1}{2}}=d(\Delta+W)^{-\frac{1}{2}}-
(\DDelta+W)^{-\frac{1}{2}}d+(\DDelta+W)^{-\frac{1}{2}}d.
$$
It follows from \eqref{crit} and  \cite[Theorem 1.1]{M} or our Theorem~\ref{th:Rp<2} that   $d^*(\DDelta+W)^{-\frac{1}{2}}$ is bounded on $L^p(\forms)$
to $L^p(M)$ for all $p\in (p'_0,2]$  with again  $p_0:=\frac{2D}{(D-2)(1-\sqrt{1-\varepsilon})}$. By duality,  $(\DDelta+W)^{-\frac{1}{2}}d$ is bounded from  $L^p(M)$ to $L^p(\forms)$ for all
$p\in [2,p_0)$.   Choosing again $\varepsilon$ small enough such that $p_0\ge p_2$,  it follows that $(\DDelta+W)^{-\frac{1}{2}}d$ is bounded on $L^p(M)$ to $L^p(\forms)$ for all $p\in [2,p_2)$.\\
It  remains  to prove the boundedness of $d(\Delta+W)^{-\frac{1}{2}}-
(\DDelta+W)^{-\frac{1}{2}}d$.
For this part we follow  the strategy in   \cite[Section 5.2]{D} and \cite[Section 3.2]{C}. In these two papers, the authors assume a global Sobolev inequality on the manifold together with a polynomial lower bound on the volume. We adapt their ideas to our setting.\\
  Let $\square_t=-\frac{\partial^2}{\partial t^2}+(\DDelta+W)$. We have for $u\in C_c^\infty(M)$
\begin{eqnarray*}
\square_t\Big(de^{-t\sqrt{\Delta+W}}u-e^{-t\sqrt{\DDelta+W}}du\Big)=\square_t de^{-t\sqrt{\Delta+W}}u= - (e^{-t\sqrt{\Delta+W}}u)dW.
\end{eqnarray*}
The operator  $L_1 :=-\frac{\partial^2}{\partial t^2}$ with domain
$$
D(L_1)=W^{2,2}((0,\infty), L^2(\forms))\cap W_0^{1,2}((0,\infty), L^2(\forms))
$$
is  self-adjoint  on $L^2((0,\infty),L^2(\forms))$.  We define  $L_2$ as an ``extension" of  $\DDelta+W$ to
$L^2((0,\infty),L^2(\forms))$ in the following usual   way
$$
(L_2 w)(t,x):=(\DDelta+W)(w(t,\cdot))(x)
$$
with  domain
$$
D(L_2):=\{w\in L^2((0,\infty),L^2(\forms)): w(t,x)\in D(\DDelta+W) \,\mbox{for a.e. t}\}.
$$
The operators  $L_1$ and $L_2$ are self-adjoint and commute. Therefore,
$e^{-sL_1}e^{-sL_2}=e^{-sL_2}e^{-sL_1}$ is a strongly  continuous  semigroup whose  generator $\mathcal{C}$ is the closure of $L_1+L_2$ on the domain $D(L_1)\cap D(L_2)$ (see for example \cite{EN}, p. 64).

Let $\phi=de^{-t\sqrt{\Delta+W}}u-e^{-t\sqrt{\DDelta+W}}du$. Assume for a moment that
 $\phi\in D(L_1)\cap D(L_2)$ and that $\mathcal{C} $ is injective.   Then
$$
\mathcal{C} \phi = \square_t\phi= - (e^{-t\sqrt{\Delta+W}}u)dW,
$$
and we have
\begin{eqnarray*}
de^{-t\sqrt{\Delta+W}}u-e^{-t\sqrt{\DDelta+W}}du&=&\phi=-\mathcal{C}^{-1}\big ((e^{-t\sqrt{\Delta+W}}u)dW\big)\\
&=&-\int_0^\infty e^{-s\mathcal{C}}((e^{-t\sqrt{\Delta+W}}u)dW)ds\\
&=&-\int_0^\infty e^{-sL_1}e^{-sL_2}((e^{-t\sqrt{\Delta+W}}u)dW)ds\\
&=&\int_0^\infty\int_0^\infty K_s(t,\sigma) e^{-s(\DDelta+W)}((e^{-t\sqrt{\Delta+W}}u)dW)d\sigma ds,
\end{eqnarray*}
where
$$
K_s(t,\sigma)=\frac{e^{-\frac{(\sigma+t)^2}{4s}}-e^{-\frac{(\sigma-t)^2}{4s}}}{\sqrt{4\pi s}}
$$
is the heat kernel on the half-line $\mathbb{R}_+$ for the Dirichlet boundary condition at $0$.

Next we write
\begin{eqnarray*}
\lefteqn{d(\Delta+W)^{-\frac{1}{2}}u-(\DDelta+W)^{-\frac{1}{2}}du}\\
&&=\int_0^\infty (de^{-t\sqrt{\Delta+W}}u-e^{-t\sqrt{\DDelta+W}}du) dt\\
&&=\int_0^\infty \int_0^\infty\int_0^\infty K_s(t,\sigma) e^{-s(\DDelta+W)}((e^{-t\sqrt{\Delta+W}}u)dW)d\sigma ds dt=:G(u).
\end{eqnarray*}

Because $W\geq 0$ and $e^{-t\Delta}$ satisfies the  Gaussian upper bound~\eqref{def:Gassianbound}, it follows  that $e^{-t(\Delta+W)}$ also satisfies the   same bound (this follows from the domination property
$| e^{-t(\Delta + W)}f | \le e^{-t\Delta} |f |$). Therefore,
\begin{eqnarray*}
\|e^{-t(\Delta+W)}\chi_{B(x_0,r_0)}\|_{p\to 2}\leq Cv(x_0,\sqrt{t})^{-(\frac{1}{p}-\frac{1}{2})} (1 + \frac{r_0}{\sqrt{t}})^D.
\end{eqnarray*}

It follows from the subordination formula $e^{-t\sqrt{A}}=\frac{1}{\sqrt{\pi}}\int_0^{\infty} \frac{e^{-u}}{\sqrt{u}}e^{-\frac{t^2 A}{4u}}du$ that
\begin{eqnarray*}
\|e^{-t\sqrt{\Delta+W}}\chi_{B(x_0,r_0)}\|_{p\to 2}\leq
Cv(x_0,t)^{-(\frac{1}{p}-\frac{1}{2})} (1 + \frac{r_0}{t})^D.
\end{eqnarray*}

By the volume condition~\eqref{eq:volumelowerbound} and the doubling condition~\eqref{def:doubling2}, we have for all $p\in[1,2]$ and $t>1$,
\begin{eqnarray}\label{eq:pto2biger2}
\|e^{-t\sqrt{\Delta+W}}\chi_{B(x_0,r_0)}\|_{p\to 2}\leq
C_{x_0,r_0}t^{-p_2(\frac{1}{p}-\frac{1}{2})},
\end{eqnarray}
and similarly for all $p\geq 2$ and $t \ge 1$
\begin{eqnarray}\label{eq:pto2biger3}
\|\chi_{B(x_0,r_0)}e^{-t\sqrt{\Delta+W}}\|_{p\to \infty}\leq
C_{x_0,r_0}t^{-\frac{p_2}{p}}.
\end{eqnarray}

From these estimates we want to obtain that $\|G(u)\|_p\leq C\|u\|_{p}$ for all $p\in[2,p_2)$.
Since these estimates are valid for $t > 1$ we have  to
treat first the case of small  $t$ and $s$ in the definition of $G$.

Let $g_{s,t}(u):=e^{-s(\DDelta+W)}((e^{-t\sqrt{\Delta+W}}u)dW)$.
Then
\begin{eqnarray*}
G(u)&=&\int_0^\infty \int_0^\infty\int_0^\infty K_s(t,\sigma)
g_{s,t}(u)d\sigma ds dt\\
&=&\int_0^\infty \int_0^\infty\int_0^\infty
\frac{e^{-\frac{(\sigma+t)^2}{4s}}-e^{-\frac{(\sigma-t)^2}{4s}}}{\sqrt{4\pi
s}} g_{s,t}(u)d\sigma ds dt\\
&=&\int_0^\infty \int_0^\infty\frac{g_{s,t}(u)}{\sqrt{4\pi
s}}\Big[\int_0^\infty
e^{-\frac{(\sigma+t)^2}{4s}}-e^{-\frac{(\sigma-t)^2}{4s}} d\sigma
\Big]ds dt\\
&=&-\int_0^\infty \int_0^\infty\frac{g_{s,t}(u)}{\sqrt{4\pi
s}}\Big[\int_{-t}^t e^{-\frac{\sigma^2}{4s}}  d\sigma \Big]ds dt\\
&=&-\int_0^\infty \int_0^\infty\frac{2g_{s,t}(u)}{\sqrt{\pi
}}\Big[\int_{0}^{\frac{t}{2\sqrt{s}}} e^{-\gamma^2}  d\gamma \Big]ds dt\\
&=&-\frac{2}{\sqrt{\pi}}\int_0^\infty \int_0^\infty
e^{-\gamma^2}\Big[\int_{0}^{\frac{t^2}{4\gamma^2}} g_{s,t}(u)ds
\Big]d\gamma dt.
\end{eqnarray*}

Thus
\begin{eqnarray}\label{eq:ptopboundforG}
\|G(u)\|_{L^p}\leq \frac{2}{\sqrt{\pi}} \int_0^\infty \int_0^\infty
e^{-\gamma^2}\Big[\int_{0}^{\frac{t^2}{4\gamma^2}}
\|g_{s,t}(u)\|_{L^p}ds \Big]d\gamma dt.
\end{eqnarray}

For all $s\in [0,1]$ and $t\in [0,1]$, by the fact that the
semigroup $e^{-s(\DDelta+W)}$ and $e^{-t\sqrt{\Delta+W}}$ are
uniformly bounded on $L^p$ for all $s>0$ and $t>0$,
\begin{eqnarray*}
\|g_{s,t}(u)\|_{L^p}&=&\|e^{-s(\DDelta+W)}((e^{-t\sqrt{\Delta+W}}u)dW)\|_{L^p}
\leq C\|dW\|_{L^\infty}\|u\|_{L^p}.
\end{eqnarray*}

For all $s\in [0,1]$ and $t>1$, by the fact that the semigroup
$e^{-s(\DDelta+W)}$ is uniformly bounded on $L^p$ for all $s>0$ and
estimate \eqref{eq:pto2biger3},
\begin{eqnarray*}
\|g_{s,t}(u)\|_{L^p}&=&\|e^{-s(\DDelta+W)}((e^{-t\sqrt{\Delta+W}}u)dW)\|_{L^p}\\
&\leq& C\|(e^{-t\sqrt{\Delta+W}}u)dW\|_{L^p}\\
&\leq&
C\|dW\|_{L^p}\|\chi_{B(x_0,r_0)}e^{-t\sqrt{\Delta+W}}u\|_{L^\infty}\\
&\leq& CC_{x_0,r_0}\|dW\|_{L^p}t^{-\frac{p_2}{p}}\|u\|_{L^p}.
\end{eqnarray*}

For all $s>1$ and $t\in [0,1]$, by the fact that the semigroup
$e^{-t\sqrt{\Delta+W}}$ is uniformly bounded on $L^p$ for all $t>0$
and estimate \eqref{eq:pto2tbiger1},
\begin{eqnarray*}
\|g_{s,t}(u)\|_{L^p}&=&\|e^{-s(\DDelta+W)}((e^{-t\sqrt{\Delta+W}}u)dW)\|_{L^p}\\
&\leq& C \|e^{-s(\DDelta+W)}\chi_{B(x_0,r_0)}\|_{(p_0'+\varepsilon)
\to p}
\|(e^{-t\sqrt{\Delta+W}}u)dW\|_{L^{p_0'+\varepsilon}}\\
&\leq&
CC_{x_0,r_0} s^{-p_2(\frac{1}{p_0'+\varepsilon}-\frac{1}{p})/2}
\|dW\|_{L^{1/(\frac{1}{p_0'+\varepsilon}-\frac{1}{p})}}\|e^{-t\sqrt{\Delta+W}}u\|_{L^p}\\
&\leq&CC_{x_0,r_0}\|dW\|_{L^{1/(\frac{1}{p_0'+\varepsilon}-\frac{1}{p})}}
s^{-p_2(\frac{1}{p_0'+\varepsilon}-\frac{1}{p})/2} \|u\|_{L^p}.
\end{eqnarray*}

Similarly, for all $s>1$ and $t>1$,
\begin{eqnarray*}
\|g_{s,t}(u)\|_{L^p}&=&\|e^{-s(\DDelta+W)}((e^{-t\sqrt{\Delta+W}}u)dW)\|_{L^p}\\
&\leq& C \|e^{-s(\DDelta+W)}\chi_{B(x_0,r_0)}\|_{(p_0'+\varepsilon)
\to p}
\|(e^{-t\sqrt{\Delta+W}}u)dW\|_{L^{p_0'+\varepsilon}}\\
&\leq& CC_{x_0,r_0}
s^{-p_2(\frac{1}{p_0'+\varepsilon}-\frac{1}{p})/2}
\|dW\|_{L^{1/(\frac{1}{p_0'+\varepsilon}-\frac{1}{p})}}\|\chi_{B(x_0,r_0)}e^{-t\sqrt{\Delta+W}}u\|_{L^p}\\
&\leq&CC_{x_0,r_0}\|dW\|_{L^{1/(\frac{1}{p_0'+\varepsilon}-\frac{1}{p})}}
s^{-p_2(\frac{1}{p_0'+\varepsilon}-\frac{1}{p})/2}\|\chi_{B(x_0,r_0)}\|_{L^{p}}
\|\chi_{B(x_0,r_0)}e^{-t\sqrt{\Delta+W}}u\|_{L^\infty}\\
&\leq&CC_{x_0,r_0}\|dW\|_{L^{1/(\frac{1}{p_0'+\varepsilon}-\frac{1}{p})}}
s^{-p_2(\frac{1}{p_0'+\varepsilon}-\frac{1}{p})/2}\|\chi_{B(x_0,r_0)}\|_{L^{p}}
t^{-\frac{p_2}{p}}\|u\|_{L^p}.
\end{eqnarray*}

Combing the above four estimates, we get
$$
\|g_{s,t}(u)\|_{L^p}\leq
C(1+t)^{-\frac{p_2}{p}}(1+s)^{-p_2(\frac{1}{p_0'+\varepsilon}-\frac{1}{p})/2}\|u\|_{L^p}.
$$

Putting this estimate into estimate \eqref{eq:ptopboundforG} and
noting that $p>2$, $p_2>3$  and $p<p_2$, we have
$$
\|G(u)\|_{L^p}\leq C\|u\|_{L^p}.
$$

Hence, $d (\Delta + W)^{-\frac{1}{2}}$ is bounded on $L^p$ for $p \in [2, p_2)$. \\

Recall that we have used in the previous  proof that $\phi\in D(L_1)\cap D(L_2)$ and $\mathcal{C}$ is injective. Now we prove
these two properties.

If $\mathcal{C}\psi =0$,  then $\psi = e^{-sL_1}e^{-sL_2} \psi =  e^{-sL_2}e^{-sL_1} \psi$ and the  self-adjointness
of $L_1$ and $L_2$ imply that $\psi \in D(L_1) \cap D(L_2)$. Hence
$$\langle\frac{\partial^2}{\partial
t^2}\psi ,\psi \rangle_{L^2(\R_+\times TM)}=0, $$
which implies  $\partial_t \psi=0$ and thus $\psi(x,t)=\psi(x)$. In addition, $\psi \in W^{1,2}_0$ implies $\psi(x, 0) = 0$ and hence $\psi = 0$. This shows that $\mathcal{C}$ is injective.

Now we prove that $\phi \in D(L_1) \cap D(L_2)$.
For fixed $t$, because $u $ and $ W$ belong to $C_c^\infty$, it is easy to see that $\phi\in L^2$ and $(\DDelta+W)\phi \in L^2$. Thus $\phi\in D(L_2)$. Note that
$$
\lim_{t\to 0} \phi=\lim_{t\to 0} de^{-t\sqrt{\Delta+W}}u-  \lim_{t\to 0}e^{-t\sqrt{\DDelta+W}}du=du-du=0.
$$
($\underset{t\to 0}{\lim}\, de^{-t\sqrt{\Delta+W}}u=du$ comes from that fact $d(\Delta+W)^{-1/2}$ is bounded on $L^2$ and $(\Delta+W)^{1/2}e^{-t\sqrt{\Delta+W}}u$ converges to $(\Delta+W)^{1/2}u$).
It remains  to check that $\phi\in W^{2,2}((0,\infty), L^2(\forms))$.
We write $\phi=\phi_1-\phi_2$ where $\phi_1:=de^{-t\sqrt{\Delta+W}}u$ and $\phi_2:=e^{-t\sqrt{\DDelta+W}}du$. Then
\begin{eqnarray*}
\int_0^\infty \|\phi_1\|_{L^2}^2 dt
&=& \int_0^\infty \|de^{-t\sqrt{\Delta+W}}u\|_{L^2}^2 dt\\
&=& \int_0^\infty \|d(\Delta+W)^{-1/2}(\Delta+W)^{1/2}e^{-t\sqrt{\Delta+W}}u\|_{L^2}^2 dt\\
&\leq &\int_0^\infty \|(\Delta+W)^{1/2}e^{-t\sqrt{\Delta+W}}u\|_{L^2}^2 dt\\
&\leq &\int_0^1 \|(\Delta+W)^{1/2}u\|_{L^2}^2 dt+\int_1^\infty \|t(\Delta+W)^{1/2}e^{-t\sqrt{\Delta+W}}u\|_{L^2}^2 \frac{dt}{t^2}\\
&\leq &\|(\Delta+W)^{1/2}u\|_{L^2}^2 +\int_1^\infty \|u\|_{L^2}^2 \frac{dt}{t^2}\\
&\leq & C.
\end{eqnarray*}

Similarly
\begin{eqnarray*}
\int_0^\infty \|\partial_t\phi_1\|_{L^2}^2 dt
&=& \int_0^\infty \|d\sqrt{\Delta+W}e^{-t\sqrt{\Delta+W}}u\|_{L^2}^2 dt\\
&=& \int_0^\infty \|d(\Delta+W)^{-1/2}(\Delta+W)e^{-t\sqrt{\Delta+W}}u\|_{L^2(X)}^2 dt\\
&\leq &\int_0^\infty \|(\Delta+W)e^{-t\sqrt{\Delta+W}}u\|_{L^2}^2 dt\\
&\leq &\int_0^1 \|(\Delta+W)u\|_{L^2}^2 dt+\int_1^\infty \|t^2(\Delta+W)e^{-t\sqrt{\Delta+W}}u\|_{L^2}^2 \frac{dt}{t^4}\\
&\leq &\|(\Delta+W)u\|_{L^2}^2 +\int_1^\infty \|u\|_{L^2}^2 \frac{dt}{t^4}\\
&\leq & C,
\end{eqnarray*}
and
\begin{eqnarray*}
\int_0^\infty \|\partial^2_t\phi_1\|_{L^2}^2 dt
&=& \int_0^\infty \|d{(\Delta+W)}e^{-t\sqrt{\Delta+W}}u\|_{L^2}^2 dt\\
&\leq &\|(\Delta+W)^{3/2}u\|_{L^2}^2 +\int_1^\infty \|u\|_{L^2}^2 \frac{dt}{t^6}\\
&\leq & C.
\end{eqnarray*}

By the same calculations, we can prove that
$\phi_2, \partial_t\phi_2, \partial^2_t\phi_2 \in  L^2((0,\infty), L^2(X))$.
This shows that  $\phi\in D(L_1)\cap D(L_2)$.\\

 \indent{\bf Step III}. We prove that $d\Delta^{-\frac{1}{2}}$ is bounded on $L^p$.
We write
\begin{eqnarray}\label{eq:decomp}
d\Delta^{-\frac{1}{2}}=(d\Delta^{-\frac{1}{2}}-d(\Delta+W)^{-\frac{1}{2}})+ d(\Delta+W)^{-\frac{1}{2}}.
\end{eqnarray}

We  have proved in the previous step that $d(\Delta+W)^{-\frac{1}{2}}$ is bounded on $L^p$
for all $p\in [2,p_2)$.  Now we  prove the boundedness of
$d\Delta^{-\frac{1}{2}}-d(\Delta+W)^{-\frac{1}{2}}$. Following the
ideas in  \cite[Section 3.6]{AO}  with $A_0=\Delta+W$ and $A=\Delta$, we write
\begin{eqnarray*}
&&d\Delta^{-\frac{1}{2}}-d(\Delta+W)^{-\frac{1}{2}}\\
&&\quad= c \int_0^\infty t^{\frac{1}{2}}d(I+tA_0)^{-1}W(I+tA)^{-1}dt\\
&&\quad= c \int_0^\infty dA_0^{-\frac{1}{2}}(tA_0)^{\frac{1}{2}}(I+tA_0)^{-\frac{1}{2}} (I+tA_0)^{-\frac{1}{2}}W^{\frac{1}{2}} W^{\frac{1}{2}}(I+tA)^{-1}dt.
\end{eqnarray*}

The operator  $dA_0^{-\frac{1}{2}}=d(\Delta+W)^{-\frac{1}{2}}$ is bounded on $L^p$ for all $p\in [2,p_2)$. Next,
$(tA_0)^{\frac{1}{2}}(I+tA_0)^{-\frac{1}{2}}$ is uniformly bounded (in $t > 0$) on $L^p$ by the holomorphic functional calculus and the fact that $A_0$ has a Gaussian bound.
For the last two terms in the previous integral, it suffices to prove that
\begin{eqnarray*}
\int_0^\infty \|W^{\frac{1}{2}}e^{-sL}\|_{p\to p}\frac{ds}{\sqrt{s}}\leq C
\end{eqnarray*}
for all $p\in [2,p_2)$, where $L$ is $A_0$ or $A$.
Noting that heat kernel of $A_0$ or $A$ satisfies Gaussian upper bound,  so by volume condition (\ref{eq:volumelowerbound}) and doubling condition, we have for all $p\in [2,p_2)$ and $t>1$,
\begin{eqnarray*}
\|\chi_{B(x_0,r_0)}e^{-tL}\|_{p\to \infty}\leq C_{x_0,r_0}t^{-\frac{p_2}{2p}}.
\end{eqnarray*}

Since $W\in C_c^\infty(M)$
\begin{eqnarray*}
\int_0^1 \|W^{\frac{1}{2}}e^{-sL}\|_{p\to p}\frac{ds}{\sqrt{s}}&\leq & C\|W\|_\infty^{\frac{1}{2}}\int_0^1 \|e^{-sL}\|_{p\to p}\frac{ds}{\sqrt{s}}\leq C\|W\|_\infty^{\frac{1}{2}}\int_0^1 \frac{ds}{\sqrt{s}}\leq C,
\end{eqnarray*}
and using $\support W\subset B(x_0,r_0)$, we deduce that for $p<p_2$
\begin{eqnarray*}
\int_1^\infty \|W^{\frac{1}{2}}e^{-sL}\|_{p\to p}\frac{ds}{\sqrt{s}}&\leq & C\|W^{\frac{1}{2}}\|_p\int_1^\infty \|\chi_{B(x_0,r_0)}e^{-sL}\|_{p\to \infty}\frac{ds}{\sqrt{s}}\\
&\leq& C\|W^{\frac{1}{2}}\|_p\int_1^\infty C_{x_0,r_0} s^{-\frac{p_2}{2p}}\frac{ds}{\sqrt{s}}\leq C.
\end{eqnarray*}

For more details about this last step, we refer to  \cite[Section 3.6]{AO}.
\end{proof}

\section{Riesz transforms of Schr\"odinger operators }
\label{sec:SchrodingerOP} \setcounter{equation}{0}

In this section, we give some results on the  boundedness of  Riesz transforms  $d A^{-\f12}$ of Schr\"odinger operators
$A=\Delta+V$ with signed potential $V=V^+-V^-$.

We start with  following result.

\begin{theorem}\label{th:Schrodinger0}
Let $M$ be a complete non-compact Riemannian manifold satisfying
assumptions \eqref{def:Gassianbound} and \eqref{def:doubling2} with
doubling dimension $D$. Let $A$ be the Schr\"odinger operator with
signed potential $V$ such that $V^+\in L_{loc}^1$ and $V^-$
satisfies $\alpha$-subcritical condition \eqref{eq:conditionV-}. Then the associated Riesz
transform $d A^{-\frac{1}{2}}$ is

i): bounded from $H_A^1(M)$ to $L^1(\forms)$,

ii): bounded from $H_A^p(M)$ to $L^p(\forms)$ for all $p\in [1,2]$,

iii): bounded on $L^p(M)$ for all $p\in(1,2]$ if $D\leq 2$ and all
$p\in (p_0',2]$ if $D>2$ where
$p_0:=\frac{2D}{(D-2)(1-\sqrt{1-\alpha})}$.
\end{theorem}

\begin{proof} Under these  assumptions on $V$ it is proved in ~\cite{AO} (see the proof of estimate (43) in page 1127) that
\begin{eqnarray}\label{eq:SingularIteCondition2}
\big\| dA^{-\frac{1}{2}}(I-e^{-r^2A})^M f\big\|_{L^2(U_j(B))}\leq C 2^{-js}
\|f\|_{L^2(B)}
\end{eqnarray}
 for every ball $B$ with    radius $r$ and for all $f\in L^2(M)$ with supp $f\subset B$. Here $M\geq 1$,  $ s>D/2$  and $C>0$ are constants.
 Now by Lemma~\ref{le:H1toL1Key} we conclude that assertion i) holds. 
 
Assertion ii) follows by interpolation between $H_A^p(M)$.

Assertion iii)  follows from ii) by identifying $H_A^p(M)$ and $L^p$ (cf. Proposition~\ref{prop:HpCoinsidesLp}) since the estimate~$(DG_{p})$ was proved in Theorem 3.4 in~\cite{AO}.
\end{proof}
Note that assertion iii) of the previous theorem was already proved in \cite{AO}.

For $p>2$, we give a consequence of Theorem \ref{th:LPp>2} and \cite[Theorem 3.9]{AO}.

\begin{theorem}\label{th:Schrodinger}
Assume that  the Riemannian manifold $M$ satisfies the doubling
condition~\eqref{def:doubling2} and the heat kernel of the Laplacian
satisfies the Gaussian upper bound~\eqref{def:Gassianbound}. Assume also
 that the negative part of the Ricci curvature $R^-$
satisfies~\eqref{eq:vol} for some $p_2 > 3$.   Let $A$ be the
Schr\"odinger operator with signed potential $V$ which satisfies ~\eqref{Vvol} and
~\eqref{eq:conditionV-} for some $\alpha \in [0, 1)$.  Then
$dA^{-\frac{1}{2}}$ is bounded on $L^p$ for $p_0'<p<\frac{p_0
r}{p_0+r}$ where $r=\inf(p_1,p_2)$.
\end{theorem}

This result is a combination of  Theorem \ref{th:LPp>2} and
\cite[Theorem 3.9]{AO}. Indeed it was proved in \cite[Theorem
3.9]{AO} that $\Delta^{\frac{1}{2}} A^{-\frac{1}{2}}$  is bounded on
$L^p$ for $p_0'<p<\frac{p_0 r}{p_0+r}$ where $r=\inf(p_1,p_2)$
without assumptions on the Ricci curvature. By Theorem
\ref{th:LPp>2}, $d\Delta^{-\frac{1}{2}}$ is bounded on $L^p$. The
boundedness of $dA^{-\frac{1}{2}}$ follows by composition.

\begin{remark}
Suppose that  $V^-=0$ (or equivalently $\alpha=0$) and $v(x,r)\ge C r^D$.
Then $p_0 = \infty$ and  the assumptions in Theorem \ref{th:Schrodinger} hold if  $R^-\in L^{\frac{D}{2}-\eta}\cap L^{\frac{D}{2}+\eta}$
and $V\in L^{\frac{D}{2}-\eta}\cap L^{\frac{D}{2}+\eta}$ for some $\eta > 0$. Thus,   the theorem gives  that  $dA^{-\frac{1}{2}}$
 is bounded on $L^p$ for all $p \in (1, D)$. Note that one cannot  expect a better interval  for boundedness of the Riesz transform of Schr\"dinger operators as we will show in the next section.
\end{remark}

\section{A negative result for the Riesz transform of Schr\"odinger operators}
\label{sec:negativeresult} \setcounter{equation}{0}

In this section, we show a negative result for the boundedness of the Riesz transform for Schr\"odinger operators. We prove even more : on a wide class of  Riemannian manifolds, the Riesz transform $d(ñ\Delta + V)^{-\frac12}$ is never bounded on $L^p$ for any $p > D$, unless eventually $V= 0$.

A result in this direction was given by Guillarmou and Hassel \cite{GH} on complete noncompact and asymptotically conic manifolds of dimension $n$. They assumed $V$ is non zero, smooth and sufficiently vanishing at infinity. They proved the Riesz transform $d(ñ\Delta + V)^{-\frac12}$ is not bounded on $L^p$ for $p>n$ if there exists a $L^2$ function $\psi$ such that $(-\Delta+V)\psi=0$.

For our concern, we recall that a Riemannian manifold  $M$ satisfies the $L^2$ Poincar\'e inequality if there exists a constant $C>0$ such that for every $f\in W^{1,2}_{loc}(M)$ and every ball  $B=B(x,r)$
\begin{equation}\label{P}
\left(\int_B|f-f_B|^2d\mu\right)^{\frac{1}{2}}\le Cr\left(\int_B|d f|^2d\mu\right)^{\frac{1}{2}},
\end{equation}
 where $\displaystyle{f_B=\frac{1}{\mu(B)}\int_B f d\mu}$.

The main result of this section is the following theorem in which we consider for simplicity only non-negative potentials.

\begin{theorem}\label{th6.1}
Assume that $M$ satisfies the doubling volume condition~\eqref{def:doubling2} and the Poincar\'e inequality $(\ref{P})$. Let $0\le V\in L^1_{loc}(M)$ and consider the Schr\"odinger operator $A=\Delta+V$. We suppose that there exists a positive function $\phi$ bounded on $M$ such that $e^{-tA} \phi = \phi$.
If $\|d e^{-tA}\|_{p-p}\le\frac{C}{\sqrt{t}}$ for some  $p> \max(D, 2)$, then $V=0$. In particular, if $dA^{-\frac12}$ is bounded
on $L^p$ for some $p > \max(D, 2)$, then $V = 0$.
\end{theorem}

\begin{remark} The assumption $e^{-tA} \phi = \phi$ for all $t\ge 0$ with $\phi$ positive bounded was studied by several authors. We give here  some references. In the Euclidean setting $M= \mathbb{R}^n$, Simon \cite{Simon} proved that if the potential $V$ is in $L^{\frac{n}{2}-\eta}\cap L^{\frac{n}{2}+\eta}$ for a certain $\eta>0$, then the assumption $e^{-tA}\phi=\phi$ for all $t\ge 0$ is equivalent to the fact that $V^-$ satisfies ($\ref{eq:conditionV-}$).  With  different methods,  Grigor'yan \cite{G} and Takeda \cite{T} proved that if $M$ is non-parabolic and satisfies Li-Yau estimates and if the potential $V$ is nonnegative and Green-bounded on $M$, then such a function $\phi$ exists.
\end{remark}

\begin{proof}[Proof of Theorem \ref{th6.1}]
 From Lemma \ref{le:Hajlasz} below, we have for all $f\in W^{1,p}(M)$ and for almost every $x,x' \in M$ with $\rho(x,x')\le 1$
\begin{equation}\label{HK}
|f(x)-f(x')|\le C_{x,p}\|d f\|_p.
\end{equation}

 Let $f\in C_c^\infty(M)$ and $x,x'\in M$ with $\rho(x,x')\le 1$ and fix $p > \max(D,2)$ such that
 $\|d e^{-tA}\|_{p-p}\le\frac{C}{\sqrt{t}}$ for all $t > 0$.
 From \eqref{HK}, we have for all $t > 0$
\begin{align}\label{666}
|e^{-tA}(v(.,\sqrt{t})^{1-\frac{1}{p}}f)(x)- e^{-tA}(v(.,\sqrt{t})^{1-\frac{1}{p}}f)(x')|
&\le C\|d e^{-tA}v(.,\sqrt{t})^{1-\frac{1}{p}}f\|_p\nonumber\\
&\le \frac{C}{\sqrt{t}}\|e^{-\frac{t}{2}A}v(.,\sqrt{t})^{1-\frac{1}{p}}f\|_p.
\end{align}
Since~\eqref{def:doubling2} and $(\ref{P})$ are equivalent to Li-Yau estimates (see \cite{SC}), the heat kernel $p_t(x,y)$ of
 $\Delta$ satisfies the Gaussian upper bound~\eqref{def:Gassianbound}. Since $V \ge 0$, the heat kernel $k_t(x,y)$ of $A$ satisfies also the same Gaussian upper bound.
As a consequence, the semigroup $e^{-tA}$ is uniformly bounded on $L^1(M)$ and the operator $e^{-tL}v(.,\sqrt{t})$ is uniformly bounded from $L^1(M)$ to $L^{\infty}(M)$. An interpolation argument shows that for all $p\in[1,\infty]$ the operator $e^{-tL}v(.,\sqrt{t})^{1-\frac{1}{p}}$ is bounded from $L^1(M)$ to $L^{p}(M)$ (see e.g.  \cite[Proposition 2.1.5]{BCS}). \\
It follows from this and \eqref{666} that
\begin{equation}
|e^{-tA}(v(.,\sqrt{t})^{1-\frac{1}{p}}f)(x)- e^{-tA}(v(.,\sqrt{t})^{1-\frac{1}{p}}f)(x')|
\le \frac{C}{\sqrt{t}}\|f\|_1.
\end{equation}
This extends by density to all $f \in L^1(M)$ and gives
\begin{equation}
\left|\int_M(k_t(x,y)-k_t(x',y))v(y,\sqrt{t})^{1-\frac{1}{p}}f(y)d\mu(y)\right|\le \frac{C}{\sqrt{t}}\|f\|_1.
\end{equation}
 Since the previous inequality is satisfied for all $f\in L^1(M)$, we obtain for a.e. $x, y\in M$ and all $t > 0$
\begin{equation}
|k_t(x,y)-k_t(x',y)|\le \frac{C}{\sqrt{t}\,v(y,\sqrt{t})^{1-\frac{1}{p}}}.
\end{equation}
Using~\eqref{def:doubling2} and~\eqref{def:Gassianbound} for $k_t(x,y)$ we find
\begin{align*}
&|k_t(x,y)-k_t(x',y)|\\
&\le|k_t(x,y)-k_t(x',y)|^{\frac{1}{2}} \left[k_t(x,y)+k_t(x',y)\right]^{\frac{1}{2}}\\
&\le \frac{C}{t^{\frac{1}{4}}v(y,\sqrt{t})^{\frac{1}{2}-\frac{1}{2p}}}\left[\frac{C}{v(x,\sqrt{t})^\frac{1}{2}}exp(-c\frac{\rho^2(x,y)}{t})+\frac{C}{v(x',\sqrt{t})^\frac{1}{2}}exp(-c\frac{\rho^2(x',y)}{t})\right]\\
&\le\frac{C}{t^{\frac{1}{4}}v(y,\sqrt{t})^{1-\frac{1}{2p}}}\left[exp(-c\frac{\rho^2(x,y)}{t})+exp(-c\frac{\rho^2(x',y)}{t})\right].
\end{align*}
Therefore,
\begin{equation}
\int_M|k_t(x,y)-k_t(x',y)|d\mu(y)\le \frac{C}{t^{\frac{1}{4}}}\left[ v(x,\sqrt{t})^{\frac{1}{2p}}+ v(x',\sqrt{t})^{\frac{1}{2p}}\right].
\end{equation}
From~\eqref{def:doubling2} and since $\rho(x,x')\le 1$, we deduce that for $t\ge 1$
\begin{equation*}
\int_M|k_t(x,y)-k_t(x',y)|d\mu(y)
\le \frac{Cv(x,\sqrt{t})^{\frac{1}{2p}}}{t^{\frac{1}{4}}}\le \frac{C v(x,1)^{\frac{1}{2p}} }{t^{\frac{1}{4}(1-\frac{D}{p})}}.
\end{equation*}
Furthermore for all $t\ge 1$
\begin{align*}
|\phi(x)-\phi(x')|
&=|e^{-tA}\phi(x)-e^{-tA}\phi(x')|\\
&=\left|\int_M(k_t(x,y)-k_t(x',y))\phi(y)d\mu(y)\right|\\
&\le \frac{C v(x,1)^{\frac{1}{2p}} }{t^{\frac{1}{4}(1-\frac{D}{p})}}\|\phi\|_{\infty}.
\end{align*}
We let $t \to +\infty$ and since $p > D$ it follows that  $\phi$ is constant on $M$. From the assumption
$e^{-tA}\phi = \phi$ it follows that $A\phi =0$. The latter equality gives  $V\phi=0$ and finally  $V=0$ since $\phi$ is positive.

Finally, since the  semigroup $e^{-tA}$ is analytic on $L^p$,  if the  Riesz transform $dA^{-\frac{1}{2}}$ is bounded on $L^p$ then  $$\|de^{-tA}\|_{p-p} = \|d A^{-\frac{1}{2}} A^{\frac{1}{2}} e^{-tA}\|_{p-p}
 \le C  \| A^{\frac{1}{2}} e^{-tA}\|_{p-p} \le   \frac{C'}{\sqrt{t}}.$$
 The previous arguments show that $V = 0$.
\end{proof}

To complete the proof of the theorem, it remains to prove the following lemma.

\begin{lemma}\label{le:Hajlasz}
Let $p\ge 2$ and $p>D$. Assume that~\eqref{def:doubling2} and $(\ref{P})$ are satisfied. For all $f\in W^{1,p}(M)$ and for almost every $x,x'\in M$ with $\rho(x,x')\le 1$ there exists a constant $C=C_{x,x',p}$ such that
\begin{equation*}
|f(x)-f(x')|\le C\|d f\|_p.
\end{equation*}
\end{lemma}

\begin{proof}
 The arguments in this  proof are  taken from \cite[p. 13-14]{HK}. We repeat them for the reader's convenience.
  Write $B_i(x)=B(x,r_i)=B(x,\frac{\rho(x,x')}{2^i})$ for each nonnegative integer $i$ and $f_B=\frac{1}{\mu(B)}\int_B f d\mu$. By the Lebesgue differentiation theorem we have for almost every $x\in M$, $f_{B_i(x)}\rightarrow f(x)$ as $i$ tends to infinity.
  Using ~\eqref{def:doubling2}, $(\ref{P})$ and  H\"older's  inequality  we obtain
\begin{align*}
|f(x)-f_{B_0(x)}|
&\le \sum_{i=0}^{\infty}|f_{B_{i+1}(x)}-f_{B_i(x)}|\\
&\le \sum_{i=0}^{\infty}\frac{1}{\mu(B_{i+1}(x))}\int_{B_{i+1}(x)}|f-f_{B_i(x)}|d\mu\\
&\le C \sum_{i=0}^{\infty}\frac{1}{\mu(B_{i}(x))}\int_{B_{i}(x)}|f-f_{B_i(x)}|d\mu\\
&\le C \sum_{i=0}^{\infty}\left(\frac{1}{\mu(B_{i}(x))}\int_{B_{i}(x)}|f-f_{B_i(x)}|^2d\mu\right)^{\frac{1}{2}}\\
&\le C \sum_{i=0}^{\infty}\frac{\rho(x,x')}{2^i}\left(\frac{1}{\mu(B_{i}(x))}\int_{B_{i}(x)}|d f|^2d\mu\right)^{\frac{1}{2}}\\
&\le C \sum_{i=0}^{\infty}\frac{1}{2^i\mu(B_{i}(x))^{\frac{1}{p}}}\|d f\|_p.
\end{align*}
Using property~\eqref{def:doubling2} and  the fact that $\rho(x,x')\le 1$,  yields
\begin{equation*}
\frac{1}{\mu(B_i(x))^{\frac{1}{p}}} \le\frac{C 2^{i\frac{D}{p}}}{ \rho(x,x')^{\frac{D}{p}}v(x,1)^{\frac{1}{p}}}.
\end{equation*}
Therefore, for $p>D$, we obtain
\begin{equation}\label{equa1}
|f(x)-f_{B_0(x)}|\le C_{x,x',p} \|d f\|_p.
\end{equation}
Similarly
\begin{equation}\label{equa2}
|f(x')-f_{B_0(x')}|\le C_{x,x',p} \|d f\|_p.
\end{equation}
Furthermore from the triangle inequality and~\eqref{def:doubling2} we have
\begin{align*}
|f_{B_0(x)}-f_{B_0(x')}|
&\le |f_{B_0(x)}-f_{2B_0(x)}|+|f_{B_0(x')}-f_{2B_0(x)}|\\
&\le \frac{C}{\mu(2B_0(x))}\int_{2B_0(x)}|f-f_{2B_0(x)}|d\mu.
\end{align*}
We use the  same arguments as above and  obtain
\begin{equation}\label{equa3}
|f_{B_0(x)}-f_{B_0(x')}|\le C_{x,x',p} \|d f\|_p.
\end{equation}
The lemma follows combining $(\ref{equa1})$, $(\ref{equa2})$ and $(\ref{equa3})$.
\end{proof}

\end{document}